\documentclass[11pt]{amsart}
\usepackage[english]{babel}
\usepackage[T1]{fontenc} \usepackage[latin1]{inputenc}
\usepackage{graphics,enumerate,amssymb,textcomp}
\usepackage{color,epsfig}

\numberwithin{equation}{section}

\theoremstyle{plain} \newtheorem{thm}{Theorem}[section]
\newtheorem{lemma}[thm]{Lemma} 
\newtheorem{proposition}[thm]{Proposition}
\newtheorem{corollary}[thm]{Corollary}
 \theoremstyle{definition}
 \newtheorem{remark}[thm]{Remark}

\def\P {\mathbb {P}}    \def\Q {\mathbb {Q}}
 \def\phi{\varphi} \def\D{\Delta} 
  \def\R{\mathbb{R}} \def\E{\mathbb{E}}
 \def\1{\mathbbm{1}}

\def\mathpal#1{\mathop{\mathchoice{\text{\rm #1}}%
    {\text{\rm #1}}{\text{\rm #1}}%
    {\text{\rm #1}}}\nolimits} 

\newcommand\Ric{\mathpal{Ric}}

\newcommand\Tr{\mathpal{Tr}}

\newcommand\trace{\mathpal{trace}}
\newcommand\dive{\mathpal{div}}
\title[]{Onsager-Machlup functional for uniformly elliptic time-inhomogeneous diffusion}

\author [K.A. Coulibaly-Pasquier]{Kol\'eh\`e A. Coulibaly-Pasquier}

\address{Nancy}
\email{kolehe.coulibaly@iecn.u-nancy.fr}

\begin{document}
\date{}
\maketitle

\begin{abstract}
  In this paper we will make the computation of the Onsager-Machlup functional of 
an inhomogeneous uniformly elliptic diffusion process. 
 This functional will have formally the same picture as in the homogeneous case, the only difference come from
  the infinitesimal variation of the volume. For example in the Ricci flow case, we find some functional which is not so far
 to the   $\mathcal L_{0} $  distance used by Lott to study this flow \cite{Lott:08}. We finish by a application 
to small ball probability for weighted sup norm, for inhomogeneous diffusion.  
\end{abstract}

\section{Introduction}
 Let $M$ be  a $n$-dimensional manifold, and a inhomogeneous uniformly elliptic operator $L_{t}$ over $M$. It is always possible to put a time dependent
 family of metrics $g(t)$ over $M$ such that 
\begin{equation}
 L_{t}  = \frac12 \D_{t} + Z(t), \label{operateur}
\end{equation}

 where $\D_{t}$ is a Laplace Beltrami operator for a metric $g(t)$ and $Z(t,.)$ is a time dependent vector field over $M$.
 Let $X_{t}(x_{0})$ a $L_{t}$-diffusion process on $M$ starting at point $x_{0}$, an example of such a diffusion could be the $g(t)$-MB introduced in
 \cite{metric}, when the family of metrics $g(t)$ come from the Ricci flow. 

Let $ d(t, x,y)$ be the Riemannian distance on $M$ according to the metric $g(t)$. Consider a smooth 
 curve $\phi : [0,T] \rightarrow M $, such that $\phi(0)= x_{0}$, we are interested in the asymptotic equivalent as $\epsilon$ goes to zero of the
 following probability 
$$ \P_{x_{0}} [ \forall t \in [0,T] \quad d(t,X_{t},\phi(t) ) \le \epsilon ].$$  

This asymptotic will depend on the product of two terms. The first one is a decreasing function of $\epsilon$ that does not depend
 on the curve and the geometries (except the dimension), and a second term that depends on the geometries around the curve $\phi$.
  This second term is expressed as a certain Lagrangian, and maximizing this term could be interpreted as finding the most probable path for the diffusion.
 This term is historically called the Onsager-Machlup functional of the diffusion $X_{t}$.

This computation will be made using the same technique
 as in the paper of Takahashi and Watanabe \cite{Tak-Wat}. We will also use the non singular drift introduced by Hara.
 Using this drift, Hara  and Takahashi in \cite{Har-Tak} have made a substantial simplification of the previous proof
 of Onsager-Machlup functional. We propose a time dependent parallel transport along a curve according to a family of
 metrics and use it to make the computation of the Onsager-Machlup functional in the inhomogeneous case.


This is the main result of the paper, here in the following theorem 
\begin{thm}\label{intro-th}
Let $X_{t}(x_{0})$ be a $L_{t}$ diffusion process starting at  point $x_{0}$ , where $L_{t} = \frac12 \D_{t} + Z(t,.)$, then we have the following asymptotic:

 $$ \P_{x_{0}} [ \forall t \in [0,T] \quad d(t,X_{t},\phi(t) ) \le \epsilon ] \sim_{\epsilon \downarrow 0} C  \exp \{- \frac{\lambda_{1} T}{\epsilon^{2}} \} \exp \{- \int_{0}^{T} H(t,\phi, \dot{\phi})\}, $$  
where $H$ is a time dependent function on the tangent bundle defined for $v\in T_{x}M$ as:
$$
\begin{aligned}
H(t, x, v) & = \frac12 \Arrowvert Z(t,x) - v \Arrowvert^{2}_{g(t)} + \frac12 \dive_{g(t)}(Z)(t,x) - \frac{1}{12} R_{g(t)} (x) \\
&+ \frac14 \trace_{g(t)} (\dot g(t)) .
\end{aligned}$$
Here $C$ , $\lambda_{1}$ are explicit constants, $\dive_{g(t)}$ and $R_{g(t)}$ are respectively the divergence operator  and  the scalar curvature  with respect to the metric $g(t)$.

\end{thm}

The paper will be organized as follows:
In the first section we will give a  parallel transport along a curve according to a family of metrics. We will use this parallel transport to get a Fermi coordinate around the smooth curve $\phi$. We will also give some local development of certain tensor that we need in the sequel.

To be self contained, in  the second section we will quickly expose some probabilistic lemmas all of them are  clearly exposed
 in the paper of Capitaine \cite{Cap}, so we will keep the same notation. In this paper \cite{Cap} the case of different norms are investigated.
 In  the literature the case of non smooth functions $\phi$  are also investigated. But we will not discuss these cases in this paper.

In the three last sections we will expose the proof of the theorem and give some related results.
 The case of $g(t)$-BM, when the family of metrics $g(t)$ comes from the Ricci flow will be investigated as an application of the theorem.
 The resulting Lagrangian gives a notion of ``space-time distance`` which presents many similarities with the $ \mathcal{L}_{0}$-distance used in the theory of Ricci flow.

\section{Parallel transport along a curve, and Fermi coordinate}

Let $\phi : [0,T] \longrightarrow M $ be  smooth curve. And consider that the manifold $M$ is endowed
  with a $C^{1}$-family of  metrics $g(t)_{t \in [0,T]}$. This family of metrics produces a time dependent family of Levi-Civita connexions that we will write $\nabla^{t}$.

If $A$ is a bilinear form over a vector space $E$, $v,w \in E$ and we have a scalar product $\langle . , . \rangle_{g(t)}$ over $E$ then we define $A^{\#g(t)} (v) \in E$ such that
 $ \langle A^{\# g(t)} v , w \rangle =  A (v,w) .$   

\begin{proposition}
 Let $ (e_{1},e_{2}, ... , e_{n})$ be an orthonormal basis of $T_{\phi(0)}M $ for the metric $g(0)$.
 Let $\tau_{t}e_{i}$ be the solutions of the following first order equation on $TM$ above the curve $\phi$:

$$\left\{ \begin{aligned} \label{parallel}
  &\nabla^{t}_{\dot \phi (t)} \tau_{t}e_{i}   = - \frac12 \dot g(t)^{\# g(t)} (\tau_{t}e_{i}) \\
 & \tau_{0}e_{i} & = e_{i}. 
 \end{aligned} \right.$$

Then  $(\tau_{t}e_{1},\tau_{t}e_{2}, ... , \tau_{t}e_{n})$ is an orthonormal basis of $T_{\phi(t)}M$ for the metric $g(t)$. 

\end{proposition}
\begin{proof}
 We have just to compute quantity like:

$$\begin{aligned}
 \frac{d}{dt} \langle \tau_{t}e_{i}, \tau_{t}e_{j} \rangle_{g(t)} 
&= \nabla^{t} g(t) ( \tau_{t}e_{i}, \tau_{t}e_{j} ) +  \langle \nabla^{t} \tau_{t}e_{i}, \tau_{t}e_{j} \rangle_{g(t)} 
+  \langle  \tau_{t}e_{i}, \nabla^{t} \tau_{t}e_{j} \rangle_{g(t)} \\
&+ \dot g(t) ( \tau_{t}e_{i}, \tau_{t}e_{j} ) \\
&= - \frac12  \langle  \dot g(t)^{\# g(t)} (\tau_{t}e_{i}), \tau_{t}e_{j} \rangle_{g(t)} 
 - \frac12 \langle  \tau_{t}e_{i},  \dot g(t)^{\# g(t)} (\tau_{t}e_{j}) \rangle_{g(t)} \\
&+ \dot g(t) ( \tau_{t}e_{i}, \tau_{t}e_{j} ) \\
&= 0 . \\
\end{aligned} $$

\end{proof}

We are now able to write the Fermi coordinate around a curve.
Let $ \phi :  [0,T] \longrightarrow M$ be a smooth curve and  let $\tau$ be the parallel transport above $\phi$ in the sense of \eqref{parallel}, where we have fix a $g(0)$ orthonormal basis  $(e_{1},...,e_{n} )$ of $T_{\phi(0)}M$. Consider the map 
$$
\begin{aligned}
 \Psi : U \subset [0,T] \times \R^{n} & \longrightarrow  V \subset [0,T] \times M \\
          (t,v_{1},...,v_{n}) &\longmapsto (t, \exp_{\phi(t)}^{t} ( \tau_{t} \sum_{1}^{n} v_{i}e_{i}) ) . \\
\end{aligned}
  $$ 
Where $\exp_{x}^{t}$ means the exponential map for the metric $g(t)$.
The map $\Psi$ is  clearly a diffeomorphism on some neighborhood $U$ of $ [0,T] \times {0}$, and let $V =  \Psi (U)$. Remark that for each fixed $t$, the map $\Psi (t, .)$ is the normal coordinate for the metric $g(t)$ around the point $ \phi(t)$.

Let $X_{t}(x_{0})$ be a $L_{t}$-diffusion that starts at the point $x_{0}$, where $L_{t}$ is a time dependent operator as is \eqref{operateur}.
 Using this Fermi coordinate, the time-dependent  norm in the theorem \ref{intro-th} will be translate in term of Euclidean one,
 but the generator will be changed by the pull back by $\Psi$ of the operator $L_{t}$.
 The generator of $(t,X_{t})$ is $ \partial_{t} + \frac12 \D_{t} + Z(t, . ) $ and we will compute the generator of $ \Psi^{-1} (t, X_{t})$,
 or more precisely its local development.

$$\begin{aligned}
 \Psi^{*} (\partial_{t} + \frac12 \D_{t} + Z(t, . )) &= \tilde{\frac{\partial}{\partial_{t}}} + \frac12 \tilde{\D}_{t} + \tilde{Z}(t,.) .\\
\end{aligned}$$
The second term in the right hand side  is computed in \cite{metric} as :
 $$ \tilde{\D}_{t}   = g^{ij}( \Psi (t,.)) \frac{\partial}{\partial x_{i}} \frac{\partial}{\partial x_{j}} - \frac12 g^{kl}(\Psi( t,.)) \Gamma_{kl}^{i}(\Psi (t,.)) \frac{\partial}{\partial x_{i}}, $$ 
where $ (x_{1}^{t},...,x_{n}^{t}) $, $g_{ij}(\Psi (t,.))$, $g^{ij}(\Psi(t,.))$, and $\Gamma_{kl}^{i}(\Psi(t,.))$  are  respectively
 the normal coordinate at the point $\phi(t)$ for the metric $g(t)$ with respect to the vector basis $ (\tau_{t}e_{1},...,\tau_{t}e_{1}) $,
 the coefficient of metric $g(t)$ in this basis, its inverse, and the Christoffel symbols of the Levi-Civita connexion of the metric $g(t)$ in this basis.
 Clearly we have 
  $$\tilde{Z}(t,.) = \sum_{i = 1}^{n} Z^{i}(t,.) \frac{\partial}{\partial x_{i}}  $$
 where $Z^{i}(t,.) = \langle Z (\Psi (t,.)), \frac{\partial}{\partial x_{i}^{t}} |_{\Psi (t,.)} \rangle_{g(t)}.$

For a point $ (t, x) \in [0,T] \times M $ in the neighborhood $V$ of $\{(t,\phi(t)), t \in [0,T]\}  $ in which $\Psi$ induce a diffeomorphism, we write $ \Psi^{-1} (t,x) = ( t, x_{1}^{t}, ... , x_{1}^{t}) \in [0,T] \times \R^{n}$. 

We have to compute $ \tilde{\frac{\partial}{\partial_{t}}}|_{(t,x)} = \sum_{i=1}^{n} a_{i}(t,x)\frac{\partial}{\partial x_{i}} + a_{0}(t,x) \frac{\partial}{\partial t}  $, clearly we have $a_{i}(t_{0},x) = \frac{\partial}{\partial_{t}}_{\rvert t_{0}} (x_{i}^{t})| _{\Psi (t_{0},x)} $.
We will compute this term using the equality :
$$\frac{\partial}{\partial t} (\exp^{g(t)} (\phi(t),\sum_{i=1}^{n}\tau_{t}e_{i}x_{i}^{t})) = 0 ,$$ 
where, for $v \in T_{x}M$,  $\exp^{g(t)} (x,v)$ is the exponential map for the metric $g(t)$ at the point $x$.
 The three propositions below   will be used to compute $a_{i}(t,x) = \frac{\partial}{\partial_{t}} (x_{i}^{t})| _{\Psi (t,x)} $.

\begin{proposition}
 Let $v \in T_{x}M$ then $$\frac{\partial}{\partial t}_{\mid_{t_{0}}} \exp^{g(t)}(x,v) = O(\| v \|^{2}_{g(t_{0})} )$$
\end{proposition}

\begin{proof}
 Let $ x_{i} (t,s) $ be the coordinate of the geodesic $\exp^{g(t)} (x,s . v)$ in the normal coordinate system centered at $ \phi(t_{0})$ with respect to the metric $g(t_{0})$, we will write $ \dot x (t,s) $ for $\frac{\partial}{\partial s } x(t,s) $.
The usual equation for the geodesic gives:
$$
\begin{aligned}
\frac{\partial}{\partial t }_{\mid_{t_{0}}} x_{i}(t,s)  &= - \frac{\partial}{\partial t }_{\mid_{t_{0}}} \Big[ \int_{0}^{s} du \int_{0}^{u} dl \, \Gamma_{jk}^{i}(t,x(t,l)) \dot x_{j}(t,l)  \dot x_{k}(t,l)  \Big] \\
                                         &= - \int_{0}^{s} du \int_{0}^{u} dl \, \Bigg( \frac{\partial}{\partial t }_{\mid_{t_{0}}} \Gamma_{jk}^{i}(t,x(t_{0},l))\dot x_{j}(t_{0},l)  \dot x_{k}(t_{0},l) \\
                                         &+ \langle d  \Gamma_{jk}^{i}(t_{0}, .) , \frac{\partial}{\partial t }_{\mid_{t_{0}}} x(t,l)\rangle \dot x_{j}(t_{0},l)  \dot x_{k}(t_{0},l) \\
                                         &+ 2 \Gamma_{jk}^{i}(t_{0},x(t_{0},l))  \frac{\partial}{\partial t }_{\mid_{t_{0}}} (\dot x_{j}(t,l))\dot x_{k}(t_{0},l) \Bigg). \\
\end{aligned}
$$
Note that $\parallel \dot x (t_{0},s) \parallel^{2}_{g(t_{0})} = \parallel v \parallel^{2}_{g(t_{0})}$,  $  \Gamma_{jk}^{i}(t_{0},x) = O (\parallel x \parallel_{g(t_{0})})$, also  $ \mid \frac{\partial}{\partial t } \Gamma_{jk}^{i}(t, . ) \mid$ and $ \parallel d  \Gamma_{jk}^{i}(t, .) \parallel$ are bounded by some constant  $ C$, in a neighborhood $V$ of $\{(t,\phi(t)), t \in [0,T]\} $, hence : 

$$\begin{aligned}
  \frac{\partial}{\partial t }_{\mid_{t_{0}}} x(t,s) & := \frac{\partial}{\partial t }_{\mid_{t_{0}}} (x_{1}(t,s),...,x_{n}(t,s)) \\
                                      &= O(\| v \|^{2}_{g(t_{0})} ) +  \int_{0}^{s} \, dl \, O(\| v \|^{2}_{g(t_{0})} )   \frac{\partial}{\partial t }_{\mid_{t_{0}}} x(t,l)\\
                                       &+  \int_{0}^{s} du \int_{0}^{u} dl \, O(\| v \|^{2}_{g(t_{0})} ) \frac{\partial}{\partial t }_{\mid_{t_{0}}} x(t,l) . \\
  \end{aligned}
 $$ 
By Gronwall's Lemma we deduce that :
$$ \parallel  \frac{\partial}{\partial t }_{\mid_{t_{0}}} x(t,1)  \parallel = O( \| v \|^{2}_{g(t_{0})}) $$

\end{proof}

\begin{proposition}
Let $(x_{1} (t) ,...,x_{n} (t)) $ be the coordinate of  $\exp_{g(t_{0})} (\phi(t),\sum_{i=1}^{n}\tau_{t}e_{i}x_{i}^{t})$
 in the normal coordinate system at the point $\phi(t_{0}) $ for the metric $g(t_{0}) $,
 and $\partial_{i} $ are the associated vectors field (it is a short notation for $ \frac{\partial}{\partial_{x_{i}^{t_{0}}}} $).
Then 
$$
\begin{aligned}
 \frac{\partial}{\partial t }_{\mid_{t_{0}}} x_{i} (t) & = \frac{\partial}{\partial t}_{\vert t_{0}} x_{i}^{t} - 
\frac12  \frac{\partial}{\partial t }_{\mid_{t_{0}}}(g (t))_{\phi (t_{0})} (\partial_{i} , \sum_{j=1}^{n} x_{j}^{t_{0}}\partial_{j}) 
+ \langle  \frac{\partial}{\partial t }_{\mid_{t_{0}}} \phi(t), \partial_{i} \rangle_{g(t_{0})}\\ 
&+ O(\parallel x ^{t_{0}} \parallel^{2}) .
\end{aligned}
$$ 
\end{proposition}

\begin{proof}
 As in the proof of the above proposition, we write the geodesic $\quad$
$\exp_{g(t_{0})} (\phi(t),s . \sum_{i=1}^{n}\tau_{t}e_{i}x_{i}^{t})$ in normal coordinate, it usually satisfy :

$$\left\{  
\begin{aligned}
 \ddot{x}_{i}(t,s) &= - \sum_{jk} \Gamma_{jk}^{i}(t_{0},x(t,s)) \dot x_{j}(t,s)\dot x_{k}(t,s) \\
  \dot x_{i}(t,0) &= \langle \sum_{l=1}^{n}\tau_{t}e_{l}x_{l}^{t}, \partial_{i\mid_{\phi(t)}}  \rangle_{g(t_{0})} \\
   x_{i}(t,0) &= \phi(t)^{i} \\
\end{aligned}
\right. $$
We have :
$$ 
\begin{aligned}
 \frac{\partial}{\partial t}_{\vert t_{0}}x_{i}(t,s) &= - \int_{0} ^{s}du \int_{0}^{u} dl \,  \sum_{jk} \frac{\partial}{\partial t}_{\vert t_{0}} \Big[ \Gamma_{jk}^{i}(t_{0},x(t,l)) \dot x_{j}(t,l)\dot x_{k}(t,l) \Big] \\
&+ s \frac{\partial}{\partial t}_{\vert t_{0}} \dot x_{i}(t,o) + \frac{\partial}{\partial t}_{\vert t_{0}} x_{i}(t,o) .\\
\end{aligned}
 $$

After the same computation as before, we could write this equality in a matrix form,
 using again the Gronwall's lemma we deduce that $\frac{\partial}{\partial t}_{\vert t_{0}} x_{i}(t,s)  $ 
is bounded for $s\in [0,1]$, so the first term in the previous computation is a $O(\lVert x^{t_{0}} \lVert ^{2}) $. Hence we have

$$ 
\begin{aligned}
 \frac{\partial}{\partial t}_{\vert t_{0}}x_{i}(t,1) &= O(\lVert x^{t_{0}} \lVert_{g(t_{0})} ^{2})  \\
&+  \frac{\partial}{\partial t}_{\vert t_{0}}  \langle \sum_{l=1}^{n}\tau_{t}e_{l}x_{l}^{t}, \partial_{i_{\mid_{\phi(t)}}}  \rangle_{g(t_{0})}  + \langle \frac{\partial}{\partial t}_{\vert t_{0}} \phi(t), \partial_{i_{\mid_{\phi(t_{0})}}}  \rangle_{g(t_{o})} .\\
\end{aligned}
 $$

Remark that $\partial_{i_{\mid_{\phi(t_{0})}}} = \tau_{t_{0}} e_{i} $ (only for $t_{0}$), so:
$$ \begin{aligned}
 \frac{\partial}{\partial t}_{\vert t_{0}}x_{i}(t,1) &= O(\lVert x^{t_{0}} \lVert_{g(t_{0})} ^{2})  
+  \frac{\partial}{\partial t}_{\vert t_{0}} x_{l}^{t} \delta_{i}^{l} \\
&+ \sum_{l=1}^{n}x_{l}^{t_{0}} \frac{\partial}{\partial t}_{\vert t_{0}}  \langle \tau_{t}e_{l}, \partial_{i_{\mid_{\phi(t)}}}  \rangle_{g(t_{0})}  + \langle \frac{\partial}{\partial t}_{\vert t_{0}} \phi(t), \partial_{i_{\mid_{\phi(t_{0})}}}  \rangle_{g(t_{o})} .\\
\end{aligned}
 $$

By the construction of the parallel transport $ \tau $ we have :
$$ \begin{aligned}
\frac{\partial}{\partial t}_{\vert t_{0}}  \langle \tau_{t}e_{l}, \partial_{i_{\mid_{\phi(t)}}}  \rangle_{g(t_{0})}&= \langle \nabla^{t_{0}}\tau_{t}e_{l}, \partial_{i_{\mid_{\phi(t_{0})}}}  \rangle_{g(t_{0})} + \langle \tau_{t_{0}}e_{l}, \nabla^{t_{0}}\partial_{i}  \rangle_{g(t_{0})} \\
&= - \frac12 \frac{\partial}{\partial t}_{\vert t_{0}} (g(t)) ( \tau_{t_{0}}e_{l} ,\partial_{i_{\mid_{\phi(t_{0})}}}  ) \\
&=  - \frac12 \frac{\partial}{\partial t}_{\vert t_{0}} (g(t)) (\partial_{l_{\mid_{\phi(t_{0})}}}  , \partial_{i_{\mid_{\phi(t_{0})}}} )
\end{aligned}
 $$ 
and the last term of the right hand side of the first equality vanishes because $ \partial _{i}$ comes from a normal coordinate 
for the metric $ g(t_{0})$. And the result follows.

\end{proof}

\begin{proposition}
 $$\frac{\partial}{\partial t}_{\vert t_{0}}x_{j}^{t} = \frac12  \frac{\partial}{\partial t }_{\mid_{t_{0}}}(g (t))_{\phi (t_{0})} (\partial_{i} , \sum_{j=1}^{n} x_{j}^{t_{0}}\partial_{j}) - \langle  \frac{\partial}{\partial t }_{\mid_{t_{0}}} \phi(t), \partial_{i} \rangle_{g(t_{0})} + O(\parallel x ^{t_{0}} \parallel^{2}) $$.
\end{proposition}
\begin{proof}
 Recall that:
$$\frac{\partial}{\partial t} (\exp^{g(t)} (\phi(t),\sum_{i=1}^{n}\tau_{t}e_{i}x_{i}^{t})) = 0 ,$$
and the two propositions above compute the first term of the previous equation.  
\end{proof}

We get the Taylor series of the generator :  
$$\begin{aligned}
\tilde{\frac{\partial}{\partial_{t}}} +  \tilde L_{t} &:=\Psi^{*} (\partial_{t} + \frac12 \D_{t} + Z(t, . ))_{\mid (t,x)} \\
&= \tilde{ \frac{\partial}{\partial_{t}}}  +  \sum_{i,j=1}^{n}g^{ij}(t,x)\frac{\partial}{\partial x_{i} }\frac{\partial}{\partial x_{j} } + \sum_{i=1}^{n} \tilde b^{i}(t,x) \frac{\partial}{\partial x_{i} }\\
&= \frac{\partial}{\partial_{t}} +  \sum_{i,j =1}^{n} \Big(\frac12 \dot g(t) \big(\frac{\partial}{\partial x_{i}^{t}},\frac{\partial}{\partial x_{j}^{t}}\big) x_{j}- \dot \phi(t)^{i} \Big) \frac{\partial}{\partial x_{i} } \\
& - \frac12 \sum_{k,l,i= 1}^{n} g^{kl}(t,x) \Gamma_{kl}^{i}(t,x) \frac{\partial}{\partial x_{i} } + \frac12 \sum_{i,j=1}^{n}g^{ij}(t,x)\frac{\partial}{\partial x_{i} }\frac{\partial}{\partial x_{j} }\\
&+ \sum_{i=1}^{n} Z^{i}(t,x) \frac{\partial}{\partial x_{i} } + O( \| x \|^{2} ) , \\
\end{aligned}$$
where $ g^{ij}(t,x) $ are the metric $g(t)$ in the normal coordinate $(x_{1}^{t}, ... , x_{n}^{t}) $ evaluated at the point $ \Psi(t,x)$, also $\Gamma_{ij}^{k}(t,x) $ are the Christoffel symbols in this coordinate at the point  $ \Psi(t,x)$,  $ \dot \phi^{i}(t)$ and $Z^{i}(t,x) $ are the coordinate of the corresponding vector in this normal coordinate.

\begin{remark}
 We have no time dependence such as  $O (\| . \|_{g(t)})$ because all the metrics are equivalent on $U$.
\end{remark}

\section{Transfer to $\R^{n}$, and probabilistic lemma}

Let $X(t) $ be a $L_{t}$ diffusion, let $ \tilde T = inf \{t\in [0,T], s.t. \, (t,X(t)) \notin V\}$ and we define $\tilde X(t) $ a process in $\R^{n}$
 such that $ ( t \wedge \tilde T, \tilde X (t) ) = \Psi^{-1} ( t \wedge \tilde T , X( t \wedge \tilde T)  $. Then for a small $\epsilon $  we have :

$$ \P_{x_{0}} [ \sup_{ t \in [0,T]} \quad d(t,X(t),\phi(t) ) \le \epsilon ] =  \P_{0} [ \sup_{ t \in [0,T]} \quad \| \tilde X(t) \|  \le \epsilon ]  .$$  

Clearly $ (t,\tilde X(t))$ is a $\tilde{\frac{\partial}{\partial t}} + \tilde L_{t}$ diffusion, so for a $\R^{n}$-valued Brownian motion $\tilde B$, $\tilde X(t)$ is a solution of the following It\^{o} stochastic differential equation :

$$\left\{ 
\begin{aligned}
 d \tilde X^i(t) &= \sum_{j = 1}^{n} \sqrt{g}^{ij} (t, \tilde X(t)) \, d\tilde B^{j}_{t} + \tilde b^i(t , \tilde X(t)) \, dt \\
   \tilde X(0) &= 0
\end{aligned}
\right.$$ 
Where $\sqrt{g}^{ij}(t, x)$ is the square root of the metric $g(t)$  in the coordinate $ (x_{1}^{t}, ... , x_{1}^{t} )$
at the point $\Psi (t, x) $ (we take the same notation for $\Gamma $)  and  

$$ \begin{aligned}
\tilde b^{i} (t, x) &= - \dot \phi ^{i}(t) - \frac12 \sum_{kl} g^{kl}(t, x) \Gamma_{kl}^{i}(t,x) \\
&+ \frac12 \dot g(t)_{\mid_{\phi (t)}} \big( \frac{\partial}{\partial x_{i}^{t}}, \sum_{j=1}^{n} x^{j} \frac{\partial}{\partial x_{j}^{t}}   \big)  
 +Z^{i} (t, x) + O( \Arrowvert x \Arrowvert^{2} ) . \\
\end{aligned} $$

We will quickly describe the Hara Besselizing drift. We have the following  equality which essentially comes from the fact that the coordinate is normal, and Gauss Lemma. For all $i\in [1..n] $ we have :
$$\sum_{j=1}^{m} g^{ij}(t, x)x_{j} = x_{i} ,  $$
$$\sum_{j=1}^{m} \sqrt{g}^{ij}(t, x)x_{j} = x_{i} .  $$
Let us define the Hara  drift :
$$ \gamma^{i} (t,x) = \frac12 \sum_{j=1}^{n } \frac{\partial g^{ij}}{\partial x_{j}} (t,x) . $$
It satisfies the following useful equation :

$$ \sum_{i=1}^{n} (1- g^{ii}(t,x)) = 2 \sum_{j}^{n} \gamma^{j} (t,x) x_{j} .$$
Let write $ \tilde \sigma _{ij}(t,.) =\sqrt{g}^{ij}(t, x) $ the unique  square root of the metric $g(t)$ in the normal coordinate. We recall the equation of $\tilde X(t)$ :
\begin{equation}
 \left\{ \label{eq_tilde}
\begin{aligned}
 d \tilde X(t) &=  \tilde \sigma (t, \tilde X(t)) \, d\tilde B{t} + \tilde b(t , \tilde X(t)) \, dt \\
   \tilde X(0) &= 0
\end{aligned}
\right. 
\end{equation}

We now define the process $Y(t) $ as a solution of the following It\^o equation 
\begin{equation}\label{equY}\left\{ 
\begin{aligned}
 d Y(t) &=  \tilde{\sigma}_{ij} (t, Y(t)) \, d\tilde B^{i}_{t} + \gamma(t , Y(t)) \, dt \\
   Y(0) &= 0 .
\end{aligned}
\right.
\end{equation} 

Using It\^o formula and the definition of the vector field $ \gamma $ we get :

$$\begin{aligned}
   d \parallel Y(t) \parallel^{2} &= 2 \sum_{k=1}^{n} Y^{k}(t) d\tilde B^{k}_{t} + n \,dt .\\
  \end{aligned}
 $$

Let $ B(t) = \sum_{k=1}^{n} \int_{0}^{t} \frac{{Y}^{k}(s)}{\parallel Y(s) \parallel} d\tilde B^{k}_{s} $, using L\'evy's Theorem,
  it is a one dimensional  Brownian motion in the filtration generated by $ \tilde B $ and 

$$\begin{aligned}
   d \parallel Y(t) \parallel^{2} &= 2  \parallel Y(t)\parallel d B_{t} + n \,dt ,\\
  \end{aligned}
 $$
 so  $\parallel Y(t) \parallel$ is a n dimensional Bessel process.
Let us  define a Girsanov's transform, such that after a change of probability, the process $Y(t)$ become equal in law to the process $\tilde X(t)$. Let 
 $$ N_{t} = \int_{0}^{t}  \langle \tilde \sigma ^{-1} (t,Y(t)) (\tilde b (t,Y(t)) - \gamma (t,Y(t)) ) , \, d\tilde B_{t} \rangle ,$$
$$ M_{t} =  \exp ( N_{t}- \frac12 \langle N \rangle_{t}) $$
$$ \Q = M_{T} . \P ,$$
by Girsanov's Theorem, $(Y, \Q)$ is a solution of (\ref{eq_tilde}). By the uniqueness in law we have :

\begin{equation} 
 \begin{aligned}
   \P_{0} [ \sup_{ t \in [0,T]} \quad \| \tilde X(t) \|  \le \epsilon ] &= \Q [\sup_{ t \in [0,T]} \quad \| Y(t) \|  \le \epsilon   ]\\
 &= \E_{\P} [ M_{T};\sup_{ t \in [0,T]} \quad \| Y(t) \|  \le \epsilon ]\\ 
&= \E_{\P} [ M_{T} \, \mid \, \sup_{ t \in [0,T]} \quad \| Y(t) \|  \le \epsilon ] \P [\sup_{ t \in [0,T]} \quad \| Y(t) \|  \le \epsilon ] .\\ 
 \end{aligned} \label{eq_proba}
\end{equation}
The term $\P [\sup_{ t \in [0,T]} \quad \| Y(t) \|  \le \epsilon ] $, is easily controled by a stopping time argument. So the problem of finding the Onsager Machlup functional become a study of the comportment of conditioned exponential martingale, as in the paper  \cite{Tak-Wat}. Let us rewrite the last term in equation (\ref{eq_proba}) as :  

\begin{equation} \label{equ1}
 \begin{aligned}
 &\E_{\P} [\exp \Bigg( \sum_{i,j=1}^{n} \int_{0}^{T} \sqrt{g}_{ij}(t,Y(t)) \delta^{j}(t,Y(t)) d\tilde B^{i}_{t}\\
&-\frac12 \sum_{i,j=1}^{n}  \int_{0}^{T} g_{ij} (t, Y(t))  \delta^{i}(t,Y(t))\delta^{j}(t,Y(t))  dt  \Bigg) \mid \, \sup_{ t \in [0,T]} \quad \| Y(t) \|  \le \epsilon] . 
 \end{aligned}
\end{equation}

Where we write $\delta^{i} (t,x) = \tilde b^{i}(t,x) - \gamma^{i}(t,x) $.

\begin{remark}
From the Lemma 1 in \cite{Cap}  it is enough to control the exponential momentum one by one in the following sense.
\end{remark}
Let us recall briefly this lemma :
\begin{lemma}[\cite{IW},\cite{Cap}]\label{lemme1}
 Let $ I_{1}, ..., I_{n}$ be $n$ random variables, $ \{A_{\epsilon}\}_{0 < \epsilon}$ a family of events, and $a_{1},...,a_{n} $ be real numbers. If, for every real number c and every $ 1 \le i \le n$,
$$\limsup_{\epsilon \to 0} \E [\exp (c I_{i}) \mid A_{\epsilon}] \le \exp(ca_{i}), $$
then,
$$\lim_{\epsilon \to 0 } \E(\exp(\sum_{i=1}^{n} I_{i}) \mid A_{\epsilon}) = \exp(\sum_{i=1}^{n}a_{i}). $$
\end{lemma}

We recall  Cartan's Theorem  concerning Taylor series of metric and curvature in normal coordinate. Note that in this case, all the metrics $g(t)$ 
are equivalent. We have :
$$g_{ij}(t,x) = \delta_{i}^{j} -\frac13 \sum_{kl} R_{iklj}(t,0)x_{k}x_{l}+ O(\| x \|^{3}), $$
where $R_{iklj}(t,0)$ are the components of the Riemannian curvature tensor, for the metric $g(t)$ in normal coordinate centered at the point $\phi (t) $.  So we deduce the following equality :
$$g^{ij}(t,x) = \delta_{i}^{j} + O(\| x \|^{2}), $$
$$\gamma^{i}(t,x) = -\frac16 \sum_{j=1}^{n} R_{ij}(t,0) x_{j}+ O(\| x \|^{2}),$$
where $R_{ij}(t,0)$ are the component of the Ricci curvature tensor, for the metric $g(t)$ in normal coordinate, at the point $\phi (t) $.
Using the definition of the Christoffel symbol we have,
\begin{equation}
 \begin{aligned}
  \Gamma_{ij}^{k}(t,x) &= \frac12 (\frac{\partial}{\partial_{x_{i}}} g_{jk}(t,x) + \frac{\partial}{\partial_{x_{j}}} g_{ik}(t,x) -\frac{\partial}{\partial_{x_{k}}} g_{ij}(t,x) ) \\
&= - \frac13 \sum_{l=1}^{n} (R_{jlik}(t,0) + R_{iljk}(t,0)) x_{l} + O( \|x \| ^{2}) .\\
 \end{aligned}
\end{equation}

So we obtain,
$$-\frac12 \sum_{i,j=1}^{n} g^{ij}(t,x) \Gamma_{ij}^{k} (t,x) = - \frac13 \sum_{l=1}^{n}R_{lk}(t,0) x_{l} + O(\|x \| ^{2}), $$
and thus,

\begin{equation}
 \begin{aligned}
  \delta^{i} (t,x) &= - \dot \phi ^{i}(t)  
+  \sum_{j=1}^{n} \big( \frac12 \dot g_{ij}(t,0) -\frac16 R_{ij}(t,0) \big) x_{j} \\ 
 &+Z^{i} (t, x) + O( \Arrowvert x \Arrowvert^{2} ) , \\
&= - \dot \phi ^{i}(t) +  Z^{i} (t, 0)
+  \sum_{j=1}^{n} \big( \frac12 \dot g_{ij}(t,0) -\frac16 R_{ij}(t,0) + \frac{\partial}{\partial x_{j}}Z^{i}(t,0) \big) x_{j} \\ 
 & + O( \Arrowvert x \Arrowvert^{2} ) , \\
 \end{aligned}
\end{equation}
\newline
where $ \dot g_{ij}(t,0) = \dot g(t) \big( \frac{\partial}{\partial x_{i}^{t}}\mid_{\phi(t)}, \frac{\partial}{\partial x_{j}^{t}}\mid_{\phi(t)} \big) . $ 
\section{Proof of the theorem}

According to Lemma \ref{lemme1} we will separately compute the terms in \eqref{equ1}. The easiest to compute is the drift term :
\begin{equation} \label{premier}
 \begin{aligned}
 & \limsup_{\epsilon \to 0 } \E[\exp \{-\frac{c}{2}  \int_{0}^{T} g_{ij} (t, Y(t))  \delta^{i}(t,Y(t))\delta^{j}(t,Y(t))  dt  \} \mid \, \sup_{ t \in [0,T]} \quad \| Y(t) \|  \le \epsilon ] \\
&\le \lim_{\epsilon \to 0} \exp [-\frac{c}{2}  \int_{0}^{T} \delta_{i}^{j} (-  \dot \phi ^{i}(t) +  Z^{i} (t, 0)^{2}  + O(\epsilon) \, dt ]\\
&\le \exp(-\frac{c}{2}  \int_{0}^{T} \delta_{i}^{j} (- \dot \phi ^{i}(t) +  Z^{i} (t, 0))^{2} \, dt),
\end{aligned}
\end{equation}
where we have used in the second inequality the fact that the $O(\epsilon) $ is uniform in $t$ according to the uniform equivalence of the family of metrics $\{g(t)\}_{t\in [0,T]}$.

In order to control the first term in \eqref{equ1} we will use the following Theorem  \cite{Har-Tak}.
\begin{thm}[\cite{Har-Tak}]\label{th-Har}
Let $\alpha$ be a one form on $[0,T] \times \R^n $, which does not depend on $dt$ and $Y(t)$ be a diffusion process in $\R^n $ whose radial part is a Bessel process, 
and $$ \langle \int_{0}^{.} Y^idY^j -Y^jdY^i , \| Y\|_{.} \rangle =0 \quad \forall i,j. $$ Then the following estimate holds for the stochastic line integral $\int_{*d(t,Y_{t})} \alpha $ (in the sens of Stratonovich integration of  a one form along a process):
$$\E [\exp (\int_{*d(t,Y_{t})} \alpha )\mid \, \sup_{ t \in [0,T]} \quad \| Y(t) \|  \le \epsilon  ] = O(\epsilon). $$
\end{thm}  
\begin{remark}
The proof of this Theorem is based on the Stokes theorem witch is transferred  to  stochastic Stokes theorem using Stratonovich integral,
 and Kunita-Watanabe theorem on orthogonal martingale theorem.
\end{remark}

To use the above Theorem we have to write the first term of \eqref{equ1} in term of Stratonovich integral of a one form along a Bessel radial part process. Using equation \eqref{equY} we get :

$$d\tilde B^{i}_{t}  = \sum_{j = 1}^{n} \tilde{\sigma}^{-1}_{ij} (t, Y(t)) \, d Y^{j}_{t} - \sum_{j = 1}^{n}\tilde{\sigma}^{-1}_{ij} \gamma^{j}(t , Y(t)) \, dt, $$

\begin{equation}
\begin{aligned}
& \sum_{i,j=1}^{n} \int_{0}^{T} \sqrt{g}_{ij}(t,Y(t)) \delta^{j}(t,Y(t)) d\tilde B^{i}_{t} \\
&= \sum_{i,j=1}^{n} \int_{0}^{T} g_{ij}(t,Y(t)) \delta^{j}(t,Y(t)) d Y^{i}_{t} - \int_{0}^{T}g_{ij}(t,Y(t)\delta^{j}(t,Y(t)) \gamma^{i}(t , Y(t)) \, dt  \\
&= \sum_{i,j=1}^{n} \int_{0}^{T} g_{ij}(t,Y(t)) \delta^{j}(t,Y(t)) *d Y^{i}_{t} \\
&-\frac12  \sum_{i,j=1}^{n} \int_{0}^{T} \langle d (g_{ij}(t,Y(t)) \delta^{j}(t,Y(t))) , d Y^{i} \rangle_{t}- \int_{0}^{T}g_{ij}(t,Y(t)\delta^{j}(t,Y(t)) \gamma^{i}(t , Y(t)) \, dt .
\end{aligned}
\end{equation}
Where $*d$ is the Stratonovich differential. 

\begin{proposition}\label{prop}
Let the event $A_{\epsilon}$ be written as $\{ \sup_{ t \in [0,T]} \quad \| Y(t) \|  \le \epsilon \}$ 
The following equalities hold, for all $i,j \in [1..n]$ and $ c\in \R$:
\begin{enumerate}[i)]
\item $$  \E[\exp (c \int_{0}^{T} \sum_{i,j=1}^{n} g_{ij}(t,Y(t)) \delta^{j}(t,Y(t)) *d Y^{i}_{t}) \mid \, A_{\epsilon} ] = O(\epsilon) $$
\item
\begin{equation*}
\begin{aligned}
& \limsup_{\epsilon \to 0}\E[\exp (-\frac{c}{2}  \int_{0}^{T} \langle d (g_{ij}(t,Y(t)) \delta^{j}(t,Y(t))) , d Y^{i} \rangle_{t}) \mid \, A_{\epsilon}  ] \\
&\le  \exp(- \frac{c}{2} \int_{0}^{T} \delta_{i}^{j}g_{ij}(t,0)\{ \frac12 \dot g_{ij}(t,0) - \frac16 R_{ij}(t,0) + \frac{\partial}{\partial x_{j}} Z^{i}(t,0) \} \, dt) \\
\end{aligned}
\end{equation*}
\item
$$ \limsup_{\epsilon \to 0}\E[ \exp(-c  \int_{0}^{T}g_{ij}(t,Y(t)\delta^{j}(t,Y(t)) \gamma^{i}(t , Y(t)) \, dt )  \mid \,A_{\epsilon} ] = 1 $$
\end{enumerate}
\end{proposition}

\begin{proof}

\begin{enumerate}[i)]

\item
Let $ \alpha = c \sum_{i,j=1}^{n} g_{ij}(t,x) \delta^{j}(t,x) dx^{i}$ in the neighborhood  $U \subset [0,T] \times \R^{n}$, and extend it in all the space. The asymptotic in the proposition is a direct consequence of Theorem \ref{th-Har}.

\item
Using It\^o formula,  equation \eqref{equY} leads to
\begin{equation*}
\begin{aligned}
& \limsup_{\epsilon \to 0}\E[\exp (-\frac{c}{2}  \int_{0}^{T} \langle d (g_{ij}(t,Y(t)) \delta^{j}(t,Y(t))) , d Y^{i} \rangle_{t}) \mid \,A_{\epsilon} ] \\
&=  \limsup_{\epsilon \to 0}\E[\exp (-\frac{c}{2}  \int_{0}^{T}  \sum_{l=1}^{n} \frac{\partial}{\partial x_{l}} (g_{ij}(t, .) \delta^{j}(t,.))(Y(t))   d Y^{l} d Y^{i} ) \mid \,A_{\epsilon}   ] \\
&=  \limsup_{\epsilon \to 0}\E[\exp (-\frac{c}{2}  \int_{0}^{T}  \sum_{l=1}^{n} \frac{\partial}{\partial x_{l}} (g_{ij}(t, .) \delta^{j}(t,.)) (Y(t))  g_{il}(t,Y(t) ) \, dt ) \mid \,A_{\epsilon} ] \\
& \le \exp(-\frac{c}{2} \delta_{i}^{j} \int_{0}^{T} g_{ij}(t, 0) (\frac12 \dot g_{ii}(t,0)- \frac16 R_{ii}(t,0)+\frac{\partial}{\partial x_{i}} Z^{i}(t,0)   ) \, dt ).
\end{aligned}
\end{equation*}
In the last computation we used the Taylor expansion that we compute in the last section.
\item 
In a similar way we have :
\begin{equation}
\begin{aligned}
&\limsup_{\epsilon \to 0}\E[ \exp(-c  \int_{0}^{T}g_{ij}(t,Y(t)\delta^{j}(t,Y(t)) \gamma^{i}(t , Y(t)) \, dt )  \mid \,A_{\epsilon} ] \\
&= \limsup_{\epsilon \to 0}\E[ \exp(-c  \int_{0}^{T} O(\| Y(t)\|)\, dt )  \mid \,A_{\epsilon} ] \\
&= 1. 
\end{aligned}
\end{equation}
\end{enumerate}

\end{proof}

\begin{proof}  Theorem \ref{intro-th}\\
Putting all things together, Lemma  \ref{lemme1} , \eqref{eq_proba}, \eqref{equ1}, \eqref{premier} and Proposition \ref{prop},  we get :
\begin{equation}
\begin{aligned}
&\lim_{\epsilon \to 0}\E_{\P} [ M_{T} \, \mid \, \sup_{ t \in [0,T]} \quad \| Y(t) \|  \le \epsilon ] \\
&=  \exp \big( \int_{0}^{T} \{-\frac12 \| Z(t,\phi(t)) - \dot \phi(t) \|^{2}_{g(t)} - \frac14 (\Tr_{g(t)}(\dot g(t)))_{\phi(t)}\\ 
& +\frac{1}{12}R(t,\phi(t)) - \frac12 \dive_{g(t)}Z (t,\phi(t)) \} dt \big) \\
&= \exp \bigg(- \int_{0}^{T} H(t,\phi(t),\dot \phi(t)) \, dt \bigg)
\end{aligned}
\end{equation}
The second term in \eqref{eq_proba} is clearly given by the scaling property of Brownian motion:
$$\P_{0} [\sup_{ t \in [0,T]} \quad \| Y(t) \|  \le \epsilon ] = \P_{0} [\tau_{1}^{n} (B) > \frac{T}{\epsilon^{2}}], $$
where $\tau_{1}^{n} (B) $ is the hitting time of the ball of radius $1$; and $B$ is a $n$ dimensional Brownian motion.  With standard argument of stopping time, Dirichlet problem and spectral Theorem we get the following : 
$$\P_{0} [\sup_{ t \in [0,T]} \quad \| Y(t) \|  \le \epsilon ] \sim_ {\epsilon \to 0} C \exp (- \lambda_{1} \frac{T}{\epsilon^{2}}) ,$$
where $ \lambda_{1}$ is the first eigenvalue of Laplace operator ($- \frac12 \D_{\R^n}  $) in the unit ball in $ \R^n$ with Dirichlet's boundary condition
 and $C$ is also an explicit constant that only depends on the dimension.
\end{proof}

\begin{equation}
\begin{aligned}
\end{aligned}
\end{equation}

\section{Small discussion about the most probable path}

 In this section we will use Theorem \ref{intro-th}, and we will give the equation of the ''most likely`` curve.
We keep the same notations as before.
%
%
%

Let $ \phi$ and $ \psi$ be two curves in $M$ such that $\phi(0) = \psi(0) $  and $\phi(T) = \psi(T) $, in the same way as in  Theorem \ref{intro-th} we have :
\\

$\begin{aligned}
& \lim_{ \epsilon \to 0 } \frac{\P_{x_{0}} [ \sup_{ t \in [0,T]} \quad d(t,X(t),\phi(t) ) \le \epsilon ]}
{\P_{x_{0}} [ \sup_{ t \in [0,T]} \quad d(t,X(t),\psi(t) ) \le \epsilon ]}  \\
&= \frac{\exp \bigg(- \int_{0}^{T} H(t,\phi(t),\dot \phi(t)) \, dt \bigg)}{\exp \bigg(- \int_{0}^{T} H(t,\psi(t),\dot \psi(t)) \, dt \bigg) } . \\ 
\end{aligned}$


 Let us compute the equation of the curve which is critical for the functional :
$$ \phi \mapsto  \int_{0}^{T} H(t,\phi(t),\dot \phi(t)) \, dt , $$
when the end point is also fixed. In the next proposition we will compute the equation of this curve in a particular case of $g(t)-Brownian \, motion$ (see \cite{metric}),
 the general case could be easily deduced by the same computation.  

\begin{proposition}
 Let $X_{t}$ be a $L_{t} := \frac 12 \D_{t}$ diffusion, where $\D_{t}$ is the Laplace operator with respect to a family of metric $g(t)$
 that come from the Ricci flow ($\partial_{t} g(t) = \alpha \Ric_{g(t)} $) as in \cite{metric}. Then the   critical curve for the functional:
 $ E: \phi \mapsto  \int_{0}^{T} H(t,\phi(t),\dot \phi(t)) \, dt  $ satisfy the following second order differential equation:

$$ \nabla_{\partial_t}^{t} \dot{\phi}(t) + \alpha \Ric^{\# g(t)} (\dot{\phi}(t)) + \frac{1-3\alpha}{12} \nabla^{t}R_{t} (\phi(t)) = 0 $$

\end{proposition}

\begin{proof}
 Let $\phi $ be a critical curve for $E$  and let $\exp$ be the exponential map according to some 
fixed metric, then for all vector field  $V$  over $\phi$ such that $ V(0) = V(1) = 0 $, we have:

 $$ \frac{\partial}{\partial_s }_ {\rvert_{s=0}} E [ t \mapsto \exp_{\phi(t)} (s V(t)) ] =0. $$
Let us recall that in this situation $ Z(t,.) = 0$, so 
$$H(t, x, v) = \frac12 \Arrowvert v \Arrowvert^{2}_{g(t)} - \frac{1-3 \alpha}{12} R_{g(t)} (x) .  $$ 
  
Let us write the  variation of the curve $\phi$ as $\phi_{V} (t,s) := \exp_{\phi(t)} (sV(t)) $,
and $ \dot{\phi_{V}} (t,s) := \frac{\partial}{\partial_t}\phi_{V} (t,s) $ then the above equation becomes:

$\begin{aligned}\label{eq_geo}
& \frac{\partial}{\partial_s}_ {\rvert_{s=0}} \int_{0}^{T} \frac12 \Arrowvert  \dot{\phi_{V}} (t,s) \Arrowvert^{2}_{g(t)} - \frac{1-3 \alpha}{12} R_{g(t)} (\phi_{V} (t,s)) = 0 \\
& = \int_{0}^{T}  \langle  \dot{\phi_{V}} (t,0) , \nabla^{t}_{\partial_s}  \dot{\phi_{V}} (t,0)  \rangle_{g(t)} \\
& -  \frac{1-3 \alpha}{12}
 \langle \nabla^{t} R_{g(t)} (\phi_{V} (t,0)) , \frac{\partial}{\partial_s} _ {\rvert_{s=0}}\phi_{V} (t,s) \rangle_{g(t)} \, dt  \\
& = \int_{0}^{T}  \langle  \dot{\phi} (t) , \nabla^{t}_{\partial_s}  \dot{\phi_{V}} (t,s)  \rangle_{g(t)}  - 
 \frac{1-3 \alpha}{12} \langle \nabla^{t} R_{g(t)} (\phi (t)) , V (t) \rangle_{g(t)} \, dt . \\
\end{aligned}$

 Since $\partial_t$ and $\partial_s$ commute, and the connection $ \nabla^t$ is torsion free we have $\nabla^{t}_{\partial_s}  \dot{\phi_{V}} (t,s)
 = \nabla^{t}_{\partial_t}  \frac{\partial}{\partial_s} \phi_{V} (t,s) $. So the characterization of the critical curve 
 become for all vector field $V$ such that $V(0)=V(T)=0$ :

\begin{equation} \label{eq__geo_2}
 \int_{0}^{T}  \langle  \dot{\phi} (t) , \nabla^{t}_{\partial_t}  V(t)  \rangle_{g(t)}  - 
 \frac{1-3 \alpha}{12} \langle \nabla^{t} R_{g(t)} (\phi (t)) , V (t) \rangle_{g(t)} \, dt =0. 
\end{equation}

 By  directs computations,
\\

$\begin{aligned}
 \partial_t \langle \dot\phi (t), V(t) \rangle_{g(t)} &= \langle \nabla^{t}_{\partial_t}  \dot\phi (t), V(t) \rangle_{g(t)} + \langle \dot\phi (t),   \nabla^{t}_{\partial_t} V(t) \rangle_{g(t)} \\
&+ \dot g(t) \bigg(    \dot\phi (t), V(t) \bigg) ,\\
\end{aligned}$

and  by the final condition of the vector field $V(0)=V(T)=0$ we have:
$$ \int_{0}^{T}   \partial_t \bigg( \langle  \dot{\phi} (t) , \nabla^{t}_{\partial_t}  V(t)  \rangle_{g(t)} \bigg) \,dt = 0 .$$

Hence equation \eqref{eq__geo_2} becomes, for all $V$ vector field such that $ V(0) = V(T) = 0 $ :

$$\int_{0}^{T}  \langle   \nabla^{t}_{\partial_t} \dot{\phi} (t) ,   V(t)  \rangle_{g(t)} + \langle \alpha \Ric^{\# g(t)} (\dot{\phi} (t)), V(t)  \rangle  + 
 \frac{1-3 \alpha}{12} \langle \nabla^{t} R_{g(t)} (\phi (t)) , V (t) \rangle_{g(t)} \, dt = 0 .  $$

We conclude that $\phi$ is a critical value of $ E$ if and only if it satisfies:
$$  \nabla^{t}_{\partial_t} \dot{\phi} (t)  + \alpha \Ric^{\# g(t)} (\dot{\phi} (t)) + \frac{1-3 \alpha}{12}  \nabla^{t} R_{g(t)} (\phi (t)) = 0 $$ 
\end{proof}

\begin{remark}
 The choice of $\alpha = \frac13$ for the speed of the backward Ricci flow, produces a simplification of 
the expression above and makes the functional $E$ positive for all time.  
\end{remark}

\begin{remark}
 In the similar way let $X_{t}$ be a $L_{t} := \frac 12 \D_{t}$ diffusion, where $\D_{t}$ is the Laplace operator with respect to a family of metric $g(t)$ then
  the $E$-critical curve $\phi$  satisfy:
$$ \nabla^{t}_{\partial_t} \dot{\phi} (t)  + \dot g(t)^{\# g(t)} (\dot{\phi} (t)) + \frac{1}{12}  \nabla^{t} R_{g(t)} (\phi (t)) 
- \frac14 \nabla^{t} (\Tr_{g(t)} \dot g(t)) (\phi (t)) = 0 . $$
We could also use this formula for the Brownian motion that come from the mean curvature flow as in \cite{process-whithout}, and compute the most probable path for 
this inhomogeneous diffusion. We could use this result to compute the most probable path for the degenerated diffusion $Z(t)$ (see Remark 2.9 of  \cite{process-whithout}).
\end{remark}

\section{Small ball properties of inhomogeneous diffusions for weighted sup norm}

 Let $ X_{t}(x) $ be a $ L_{t}  = \frac12 \D_{t} + Z(t)$ diffusion, with the same notation as in the introduction. 
 Let $ f \in C^{1} ([0,T])  $ which we assume to be a positive function on $ [0,T]$, we want to estimate the following probability 

$$ \P_{x_{0}} [ \forall t \in [0,T] \quad d(t,X_{t},\phi(t) ) \le \epsilon f(t) ],$$ 
when $\epsilon$ is closed  to $0$.

\begin{proposition} 
$$
\begin{aligned}
\P_{x_{0}} [ \forall t \in [0,T] \quad d(t,X_{t},\phi(t) ) \le \epsilon f(t) ] \sim_{\epsilon \downarrow 0}  \\
C  \exp \{- \frac{\lambda_{1} \int_{O}^{T} \frac{1}{f^{2}(s)} \, ds }{\epsilon^{2}} \} \exp \{- \int_{0}^{T} \tilde{H}(t,\phi, \dot{\phi})\} \, 
\end{aligned}$$
where
\\

 $\begin{aligned}
 \tilde{H}(t,x,v) =   \lVert Z(t,x) - v  \lVert^{2}_{g(t)} + \frac12   \dive_{g(t)}(Z)(t,x) - \frac{1}{12}  R_{g(t)} (x)  \\
+ \frac14  f^{-2} (t) \trace_{g(t)} (\dot g(t)) - \frac12 n (f'(t) f^{-3}(t)) .  
  \end{aligned}$
\end{proposition}
\begin{proof}
 Let $ \tilde{g} (t) =  \frac{1}{f^2(t)} g(t)$, and $\tilde{d} (t, ., .)$ the associated distance. Then the probability that we have to estimate is
$$ \P_{x_{0}} [ \forall t \in [0,T] \quad \tilde{d}(t,X_{t},\phi(t) ) \le \epsilon]. $$
Now after a change of time we will transform $X$ the $L_t$ diffusion to a $\tilde{L}(t) $ diffusion,  in order to use Theorem \ref{intro-th}.
Let 
$$ \delta (t) = \left(\int_{0}^{ .} \frac{1}{f^{2}(s)} \, ds \right)^{-1}(t),$$
and let $ \tilde{X}(t) := X_{\delta{(t)}}$, then $ \tilde{X}$ become a  $\tilde{L}(t)$ diffusion, where 
$$\tilde{L}(t) := \frac12 \D_{\tilde{g}(\delta{(t)})} + f^{2}(\delta{(t)}) Z(\delta{(t)}, .).$$  

We deduce that:
\\

$\begin{aligned}
&\P_{x_{0}} [ \forall t \in [0,T] \quad d(t,X_{t},\phi(t) ) \le \epsilon f(t) ] = \\
&\P_{x_{0}} [ \forall t \in [0,T] \quad \tilde{d}(t,X_{t},\phi(t) ) \le \epsilon] = \\
&\P_{x_{0}} [ \forall t \in [0,\delta^{-1} (T) ] \quad \tilde{d}(\delta(t),\tilde{X}_{t },\phi(\delta (t)) ) \le \epsilon] =\\
& \sim_{\epsilon \downarrow 0}  C  \exp \{- \frac{\lambda_{1} \delta^{-1}(T)}{\epsilon^{2}} \} \exp \{- \int_{0}^{\delta^{-1}(T)} H(\delta(t),\phi(\delta (t)), \dot{\delta}(t)\dot{\phi}(\delta (t)) \, dt\}, 
\end{aligned}$

where in the last line we have used Theorem \ref{intro-th}, also the Lagrangian $ H$ in the last equation is related to the diffusion $\tilde{X}$.
After a change of variable we get the proposition.
\end{proof}

We  directly deduce the following small ball estimate:  
\begin{corollary}
$$ \epsilon^{2} \log \{  \P_{x_{0}} [ \forall t \in [0,T] \quad d(t,X_{t},\phi(t) ) \le \epsilon f(t) ] \} \rightarrow_{\epsilon \to 0} -\lambda_{1} \int_{O}^{T} \frac{1}{f^{2}(s)} \, ds  $$
\end{corollary}






%

\end{document}